\documentclass[reqno,11pt]{amsart}
\usepackage{latexsym}
\usepackage{amsfonts}
\usepackage{amssymb}
\usepackage{mathrsfs}
\usepackage[all]{xy}
\usepackage{bbm}
\usepackage{color}
\usepackage[mathscr]{euscript}
\usepackage{tikz} 

\newtheorem{theorem}{Theorem}[section]
\newtheorem{corollary}[theorem]{Corollary}
\newtheorem{lemma}[theorem]{Lemma}

\newtheorem{proposition}[theorem]{Proposition}
\theoremstyle{definition}

\newtheorem{definition}[theorem]{Definition}
\newtheorem{remark}[theorem]{Remark}
\numberwithin{equation}{section}

\newcommand{\RR}{\mathbb{R}}

\newcommand{\bP}{\mathbf{P}}

\newcommand{\bL}{\mathbf{L}}

\newcommand{\eps}{\varepsilon}

\newcommand{\bfW}{\mathbf{W}}

\newcommand {\ignore}[1]  {}

\newcommand{\btheta}{{\boldsymbol{\theta}}}

\begin{document}

\title{Hausdorff and packing dimension of Diophantine sets}
\author{Antoine Marnat}
\date{}
\dedicatory{to Patrick Berguin}

\maketitle

Using the variational principle in parametric geometry of numbers, we compute the Hausdorff and packing dimension of Diophantine sets related to exponents of Diophantine approximation, and their intersections. In particular, we extend a result of Jarn\'ik and Besicovitch to intermediate exponents.\\

\paragraph{{\bf Aknowledgment}} The author is supported by Austrian Science Fund (FWF), Project I 3466-N35 and EPSRC Programme Grant EP/J018260/1.

\section{Exponents of Diophantine approximation and Diophantine sets}
Throughout this paper, the integer $n\geq1$ denotes the dimension of the ambient space. We denote by $\btheta=(\theta_1, \ldots , \theta_n)$ an $n$-tuple of real numbers such that $1,\theta_1, \ldots , \theta_n$ are $\mathbb{Q}$-linearly independent.\\

We first define exponents of Diophantine approximation, and some of their properties. Then, we define their related Diophantine sets, and state our theorems. In Section \ref{PGN}, we shortly describe the theory of \emph{parametric geometry of numbers}, and in particular the \emph{variational principle}. The latter is the main tool for the proofs of Section \ref{proofs}.

\subsection{Exponents of Diophantine approximation} Let $d$ be an integer with $0\leq d \leq n-1$. We define the \emph{ordinary exponent} ${\omega}_{d}(\btheta)$  (resp. the \emph{uniform exponent} $\hat{\omega}_{d}(\btheta)$) as the supremum of the real numbers $\omega$ for which there exist  rational affine subspaces $L \subset \mathbb{R}^{n}$ such that 
 \[ \dim(L)=d \; ,  \; H(L)\leq H \; \textrm{ and } \; H(L)d(\btheta,L) \leq H^{-\omega}  \]
  for arbitrarily large real numbers $H$  (resp. for every sufficiently large real number $H$). Here $H(L)$ denotes the height of $L$ (see \cite{SchH} for more details), and $d(\btheta,L)=\min_{P\in L} d(\btheta,P)$ is the minimal distance between $\btheta$ and a point of $L$. \\

These exponents were introduced originally by Laurent \cite{MLwd} and already studied implicitly by Schmidt \cite{SchH}. 
They interpolate between the classical exponents $\omega(\btheta)=\omega_{n-1}(\btheta)$ and $\lambda(\btheta)=\omega_0(\btheta)$ (resp. $\hat{\omega}(\btheta)=\hat{\omega}_{n-1}(\btheta)$ and $\hat{\lambda}(\btheta)=\hat{\omega}_0(\btheta)$) that were  introduced by Khintchine \cite{Khin1,Khin2}, Jarn\'ik \cite{JAR} and Bugeaud and Laurent \cite{BugLau,BugLau2}. Hence, for $1\le d \le n-2$ we sometime call these exponents the \emph{intermediate exponents}.\\
 
 We have the straightforward relations
 \[ \omega_{0}(\btheta) \leq \omega_{1}(\btheta) \leq \cdots \leq \omega_{n-1}(\btheta) ,  \]
 \[\hat{\omega}_{0}(\btheta)\leq \hat{\omega}_{1}(\btheta) \leq \cdots \leq \hat{\omega}_{n-1}(\btheta),\]
 and Minkowski's First Convex Body Theorem \cite{Mink} and Mahler's compound convex bodies theory provides the lower bounds
 
 \begin{equation*}
 \omega_{d}(\btheta) \geq \hat{\omega}_{d}(\btheta) \geq \frac{d+1}{n-d}, \; \textrm{  for  } 0 \leq d \leq n-1 .
 \end{equation*}
 
 Khintchine \cite{Khin2} was the first to notice that these exponents are related, with his transference principle between $\omega_0$ and $\omega_{n-1}$.
 In \cite{MLwd}, Laurent provides a splitting of \emph{Khintchine's Transference Principle} in terms of the intermediate exponents $\omega_d$. In \cite{Royspec}, Roy shows that it is optimal.

\begin{theorem}[Laurent -- Roy]\label{GUGD}
Let $n\geq2$, for all $n$-tuple of real numbers $\btheta=(\theta_{1}, \ldots, \theta_{n})$ with $1, \theta_{1}, \ldots  , \theta_{n}$ $\mathbb{Q}$-linearly independent, the exponents satisfy \[\omega_{0}(\btheta)\geq \frac{1}{n}\] and the following relations for $1\leq d \leq n-1$:
\begin{equation}\label{gugd}  \frac{d\omega_{d}(\btheta)}{\omega_{d}(\btheta)+d+1} \leq \omega_{d-1}(\btheta) \leq \frac{(n-d)\omega_{d}(\btheta)-1}{n-d+1},\end{equation}
where the left hand side is $d$ if $\omega_{d}(\btheta)=+\infty$.\\
Furthermore, for all $n$-tuple $\omega_{0}, \ldots, \omega_{n-1} \in [0, +\infty]$, satisfying $\omega_{0}\geq \frac{1}{n}$ and
\[  \frac{d\omega_{d}}{\omega_{d}+d+1} \leq \omega_{d-1} \leq \frac{(n-d)\omega_{d}-1}{n-d+1},\]
there exists a $n$-tuple $\btheta=(\theta_{1}, \ldots, \theta_{n})$ of real numbers with $1, \theta_{1}, \ldots , \theta_{n}$ $\mathbb{Q}$-linearly independent such that for $0\leq d \leq n-1$, one has \[\omega_{d}(\btheta)=\omega_{d}.\]
\end{theorem}

In particular, if the equality $\omega_d = \frac{d+1}{n-d}$ holds for some $0\le d \le n-1$, it holds for all $0\le d \le n-1$. Following the denomination of \cite{MLwd} and \cite{SchH}, we call \emph{going up} and \emph{going down} the left and right hand-side of \eqref{gugd}.\\

\subsection{Diophantine sets}

For $0\leq d \leq n-1$, we define the \emph{Diophantine sets} $\bfW_n^d({\omega}_{d})$ and $\bfW_n^d({\omega}_{d})^*$ related to $\omega_d$  by 
\begin{eqnarray*}
\bfW_n^d({\omega}_{d})&=&\{\btheta\in \RR^n | {\omega}_{d}(\btheta)\geq{\omega}_{d}  \},\\
\bfW_n^d({\omega}_{d})^*&=&\{\btheta\in \RR^n | {\omega}_{d}(\btheta)={\omega}_{d}  \}.\\
\end{eqnarray*}

We are going to consider the Hausdorff and packing dimensions of these sets and their intersections. Hausdorff dimension was introduced by F. Hausdorff \cite{Haus}, while the packing dimension was introduced by Tricot \cite{Tric}. For basics and more on these notions, see \cite{Falc} and \cite{BisPer}.\\

Jarn\'ik \cite{JarH} and Besicovitch \cite{Bes} obtained the Hausdorff dimension of Diophantine sets when $d=0$ and $d=n-1$.

\begin{theorem}[Jarn\'ik -- Besicovitch]\label{JB}
For $n\ge 1$, we have
\begin{eqnarray*}
\dim_H(\bfW_n^0({\omega}_{0})^*) &=& \dim_H(\bfW_n^0({\omega}_{0})) \quad = \quad \frac{n+1}{1+{\omega}_{0}},\\
\dim_H(\bfW_n^{n-1}({\omega}_{{n-1}})^*) &=& \dim_H(\bfW_n^{n-1}({\omega}_{{n-1}})) \quad = \quad {n-1}+ \frac{n+1}{1+{\omega}_{{n-1}}}.
\end{eqnarray*}
\end{theorem}
Note that Bovey and Dodson extended the result to the case of matrices \cite{Bodod}.\\

We now provide new results concerning Hausdorff and packing dimensions of Diophantine sets,  extending Theorem \ref{JB} to the intermediate exponents.

\begin{theorem}\label{Hwd}
For any $0\le d \le n-1$, the Hausdorff dimension of the Diophantine set $\bfW_n^d({\omega}_{d})$ is
\[ \dim_H\left(\bfW_n^d({\omega}_{d})^*\right) = \dim_H\left(\bfW_n^d({\omega}_{d})\right) = d+ \frac{n+1}{1+{\omega}_{d}}.\]
 In particular, the Hausdorff dimension is full if and only if ${\omega}_d(\btheta) = \frac{d+1}{n-d}$.
\end{theorem}

For $d=0$ and $d=n-1$ this is Theorem \ref{JB}.\\

It appears that the packing dimension of Diophantine sets is full in a strong sense.

\begin{theorem}\label{Pcap}
Fix the dimension $n\ge 2$, choose a $n$-tuple of real numbers $\omega_{0}, \ldots , \omega_{n-1}$ satisfying the \emph{going-up} and \emph{going down} relations \eqref{gugd}. Then
\begin{eqnarray}
\dim_P\left( \bigcap_{d=0}^{n-1} \bfW_n^d({\omega}_{d})^*\right) &=& \dim_P\left( \bigcap_{d=0}^{n-1} \bfW_n^d({\omega}_{d}) \right) =n.
\end{eqnarray}
In particular, $ \dim_P(\bfW_n^d({\omega}_{d})^*) = \dim_P(\bfW_n^d({\omega}_{d}))=n$.
\end{theorem}
Here and after, the condition  that $\omega_{0}, \ldots , \omega_{n-1}$ satisfy \eqref{gugd} ensures that our intersections are not empty.

Finally, we provide the exact Hausdorff dimension of intersections of Diophantine sets.

\begin{theorem}\label{Hcap}
Fix the dimension $n\ge 2$, choose a $n$-tuple of real numbers $\omega_{0}, \ldots , \omega_{n-1}$ satisfying the \emph{going-up} and \emph{going down} relations \eqref{gugd}.  Then
\begin{eqnarray}
 \dim_H\left( \bigcap_{d=0}^{n-1} \bfW_n^d({\omega}_{d})^*\right) &=& \dim_H\left( \bigcap_{d=0}^{n-1} \bfW_n^d({\omega}_{d}) \right) \quad = \quad  2\left( \sum_{d=0}^{n-1 }\frac{1}{1+\omega_d}\right).
\end{eqnarray}
In particular, the Hausdorff dimension is full if and only if $\omega_d= \frac{d+1}{n-d}$ for all $0\le d \le n-1$.\
\end{theorem}

The Hausdorff dimension is a decreasing function of $\omega_d$ for all $0\le d \le n-1$. The \emph{going-up} and \emph{going down} relations \eqref{gugd} provide that for any $0\le d < d' \le n-1$,
\begin{eqnarray*}
 \frac{1}{1+\omega_{d'}} &\le& \frac{n-d'}{n-d} \cdot \frac{1}{1+\omega_{d}},\\[3mm]
  \frac{1}{1+\omega_d} &\le& \frac{d'-d}{d'+1} + \frac{d+1}{d'+1} \cdot \frac{1}{1+\omega_{d'}}.
\end{eqnarray*}
We can use these formulae to compute the Hausdorff dimension of the sets $\left( \bigcap_{d\in\mathcal{D}_n} \bfW_n^d({\omega}_{d}) \right)$ where $\mathcal{D}_n$ is any subset of $\{0, \ldots , n-1\}$. In particular, we recover Theorem \ref{Hwd} and we have 
\begin{corollary}
\[  \dim_H\left( \bfW_n^0({\omega}_{0})^* \cap \bfW_n^{n-1}({\omega}_{n-1})^*\right) =\frac{n(2+\omega_{0}+\omega_{n-1})}{(1+\omega_{0})(1+\omega_{n-1})}.\]
\end{corollary}
Note that in this case, as expected, the Hausdorff dimension is bounded above by \[\dim_H\left(\bfW_n^{n-1}({\omega}_{n-1})^*\right) = n-1 + \frac{n+1}{\omega_{n-1} +1} \textrm{  and } \dim_H\left(\bfW_n^{0}({\omega}_{0})^*\right) = \frac{n+1}{\omega_{0} +1}.\] We have equality when $\omega_0$ is respectively the upper and lower bound in Khintchine's transference principle.\\

It is an interesting open question to extend these results to the case of intermediate exponents of matrices, or to Diophantine sets related to uniform exponents. See the related definitions in \cite{Ger}, and open questions in \cite[\S3]{VarPrinc}.\\

\section{Variational principle in parametric geometry of numbers}\label{PGN}
 
We use the notation introduced by Roy in \cite{Roy,Royspec} which is essentially dual to the one of Schmidt and Summerer \cite{SSfirst,SS}. We refer the reader to these papers for further details. Here $ \boldsymbol{x}\cdot \boldsymbol{y}=x_1y_1+ \cdots + x_ny_n$  is the usual scalar product of vectors $\boldsymbol{x}$ and $\boldsymbol{y}$, and $\|\boldsymbol{x}\|_2=\sqrt{\boldsymbol{x}\cdot \boldsymbol{x}}$ is the usual Euclidian norm.\\

Let $\boldsymbol{u} = (u_0, \ldots , u_n)$ be a vector in $\mathbb{R}^{n+1}$, with Euclidean norm $\|\boldsymbol{u}\|_2=1$. For a real parameter $Q\geq1$ we consider the convex body
\[ \mathcal{C}_{\boldsymbol{u}}(Q) = \left\{ \boldsymbol{x} \in \mathbb{R}^{n+1} \mid \|\boldsymbol{x}\|_2 \leq 1 , \; |\boldsymbol{x} \cdot \boldsymbol{u} | \leq Q^{-1}     \right\}. \]
 For $1\leq d \leq n+1$ we denote by $\lambda_d\left( \mathcal{C}_{\boldsymbol{u}}(Q) \right)$ the $d$-th minimum of $\mathcal{C}_{\boldsymbol{u}}(Q)$ relatively to the lattice $\mathbb{Z}^{n+1}$. For $q\geq0$ and $1\leq d \leq n+1$ we set
\[ L_{\boldsymbol{u},d}(q) = \log   \lambda_d\left( \mathcal{C}_{\boldsymbol{u}}(e^q) \right). \]
Finally, we define the following map associated with $\boldsymbol{u}$:

\[ \begin{array}{rccl}
 \boldsymbol{L_u} :& [0,\infty ) & \to& \mathbb{R}^{n+1}  \\
  & q & \mapsto& (L_{\boldsymbol{u},1}(q), \ldots , L_{\boldsymbol{u},n+1}(q))  .
  \end{array}\]
  The lattice $\mathbb{Z}^{n+1}$ is invariant under permutation of coordinates. Hence,  $\boldsymbol{L_u}$ remains the same if we permute the coordinates in $\boldsymbol{u}$. Since $\|\boldsymbol{u}\|_2=1$ we can thus assume that $u_0\neq0$. \\
  
The following proposition links the exponents of Diophantine approximation associated with $\boldsymbol{\theta}=(\frac{u_1}{u_0}, \ldots , \frac{u_n}{u_0})$ to the behavior of the map $ \boldsymbol{L_u}$, assuming $u_0\neq0$. It was first stated by Schmidt and Summerer in \cite{SSfirst} (Theorem 1.4). It also appears as Relations (1.8) and (1.9) in \cite{SS}. In the notation of Roy \cite{Royspec} (Proposition 3.1), it reads as follows.

\begin{proposition}\label{prop}
Let $\boldsymbol{u} = (u_0, \ldots , u_n) \in \mathbb{R}^{n+1}$, with Euclidean norm $\|\boldsymbol{u}\|_2=1$ and  $u_0\neq 0$. Set $\boldsymbol{\theta}=(\frac{u_1}{u_0}, \ldots , \frac{u_n}{u_0})$. We have the following relations:
\begin{eqnarray*}
\liminf_{q\to+\infty}  \frac{ L_{\boldsymbol{u},1}(q) + \cdots + L_{\boldsymbol{u},k}(q)}{q} &=& \frac{1}{1+{\omega}_{n-k}(\boldsymbol{\theta})},\\
\limsup_{q\to+\infty} \frac{ L_{\boldsymbol{u},1}(q) + \cdots + L_{\boldsymbol{u},k}(q)}{q} &=& \frac{1}{1+\hat{\omega}_{n-k}(\boldsymbol{\theta})}.\\
\end{eqnarray*}
\end{proposition}

Thus, if we know an explicit map $\boldsymbol{P}=(P_1, \ldots, P_{n+1}): [0,\infty) \to \mathbb{R}^{n+1}$, such that 
$\boldsymbol{L}_{\boldsymbol{u}}-\boldsymbol{P}$ is bounded, then we can compute the $2n$ exponents $\hat{\omega}_{0}(\boldsymbol{\theta}), \ldots , \hat{\omega}_{n-1}(\boldsymbol{\theta}), {\omega}_{0}(\boldsymbol{\theta}), \ldots , {\omega}_{n-1}(\boldsymbol{\theta})$ for the above point $\boldsymbol{\theta}$ upon replacing $L_{\boldsymbol{u},d}$ by $P_d$ in the above formulas for $1\leq d \leq n$.\\
For this purpose, we consider the following family of maps, introduced by Roy in \cite{Royspec}.

\begin{definition}[Roy, 2014]
Let $I$ be a subinterval of $[0,\infty)$ with non-empty interior. A Roy-system of dimension $n+1$ on $I$ is a continuous piecewise linear map $\boldsymbol{P} = (P_1, \ldots , P_{n+1}): I \to \mathbb{R}^{n+1}$ with the following three properties.

\begin{description}

\item[(S1)]{For each $q\in I$, we have $0\leq P_1(q) \leq \cdots \leq P_{n+1}(q) $  and $P_1(q) + \cdots + P_{n+1}(q) =q$.}

\item[(S2)]{ If $H$ is a non empty open subinterval of $I$ on which $\boldsymbol{P}$ is differentiable, then there are integers $\underline{r}, \bar{r}$ with $1\leq \underline{r} \leq \bar{r} \leq n+1$ such that $P_{\underline{r}}, P_{\underline{r}+1}, \ldots , P_{\bar{r}}$ coincide on the whole interval $H$ and have slope $1/(\bar{r}-\underline{r}+1)$ while any other component $P_d$ of $\boldsymbol{P}$ is constant on $H$ .}

\item[(S3)]{ If $q$ is an interior point of $I$ at which $\boldsymbol{P}$ is not differentiable, if $\underline{r}, \bar{r}, \underline{s},\bar{s}$ are the integers for which
\[ P_d'(q^-) = \frac{1}{\bar{r}-\underline{r}+1}  \;  (\underline{r} \leq d \leq \bar{r}) \; \textrm{  and  }  \; P_d'(q^+) = \frac{1}{\bar{s}-\underline{s}+1}  \;  (\underline{s} \leq d \leq \bar{s}) \; , \]
and if $\underline{r}<\bar{s}$, then we have $P_{\underline{r}}(q) = P_{\underline{r}+1}(q) = \ldots = P_{\bar{s}}(q)$.
 }\\
\end{description}
\end{definition}

Here $P_d'(q^-)$ (resp. $P_d'(q^+)$) denotes the left (resp. right) derivative of $P_d$ at $q$.
The next fondamental result combine Theorem 4.2 and Corollary 4.7 of \cite{Royspec}.

\begin{theorem}[ Roy, 2014]\label{DR}
For each non-zero point $\boldsymbol{u} \in \mathbb{R}^{n+1}$, there exists $q_0\geq 0$ and a Roy-system $\boldsymbol{P}$ on $[q_0,\infty)$ such that $\boldsymbol{L_u} - \boldsymbol{P}$ is bounded on $[q_0,\infty)$. Conversely, for each Roy-system $\boldsymbol{P}$ on an interval $[q_0,\infty)$ with $q_0\geq0$, there exists a non-zero point $\boldsymbol{u}\in\mathbb{R}^{n+1}$ such that $\boldsymbol{L_u} - \boldsymbol{P}$ is bounded on $[q_0,\infty)$.\\
\end{theorem}

In a landmark paper, Das, Fishman, Simmons and Urba\'nski \cite{VarPrinc} provide a quantitative version of Roy's Theorem, and extend it to the case of matrices. We state it here in the particular case of vectors, but first we need some definitions.\\

For a Roy-system $\bP$ on $[Q_0,\infty)$, we define the \emph{local contraction rate} $\delta(\bP,q)$ by 
\[\delta(\bP,q) =n+1-\kappa\]
where $\kappa= \min\{ k \mid P'_k(q)>0 \}$. We then consider the \emph{average contraction rate} defined by 
\[ \Delta(\bP,Q)= \frac{1}{Q-Q_0}\int_{Q_0}^Q \delta(\bP,q) dt,\]
and consider the lower and upper average contraction rates $\underline{\delta}(\bP)$ and $\overline{\delta}(\bP)$ defined by
\begin{eqnarray}\label{ud}
\underline{\delta}(\bP)&=& \liminf_{Q\to\infty}\Delta(\bP,Q),\\ \label{od}
\overline{\delta}(\bP)&=& \limsup_{Q\to\infty}\Delta(\bP,Q).
 \end{eqnarray}

\begin{theorem}[Variational principle, Das -- Fishman -- Simmons -- Urba\'nski, 2019]
Be $\mathcal{P}$ a set of Roy-systems on $[Q_0,\infty)$ closed under finite perturbation. Let 
\[\mathcal{M}(\mathcal{P}) = \{ \btheta \in \RR^n \mid \exists \bP \in \mathcal{P}, C\in \RR, \| L_\btheta-\bP\| \le C \}.\]
Then,
\begin{eqnarray*}
\dim_H(\mathcal{M}(\mathcal{P})) = \sup_{\bP \in \mathcal{P}} \underline{\delta}(\bP),\\
\dim_P(\mathcal{M}(\mathcal{P})) = \sup_{\bP \in \mathcal{P}} \overline{\delta}(\bP).\\
 \end{eqnarray*}
\end{theorem}


\section{Proof of Theorems on dimensions of intersection of Diophantine sets}\label{proofs}

In this section, we prove the results concerning the Hausdorff and packing dimensions of Diophantine sets and their intersections. For this, we apply the \emph{variational principle}. 

\subsection{Idea of the proofs}

Consider $\mathcal{F}$ the set of Roy-systems at finite distance of the set $\left\{ \bL_{\btheta} \mid \btheta \in \bigcap_{d=0}^{n-1} \bfW_n^d({\omega}_{d})^* 
 \right\}$. By the remark following Proposition \ref{prop}, it is closed under finite perturbation. The variational principle provides that:

\begin{eqnarray*}
\dim_H\left(  \bigcap_{d=0}^{n-1} \bfW_n^d({\omega}_{d})^*
  \right) &=& \sup_{\bP \in \mathcal{F} } \underline{\delta}(\bP),\\
 \dim_P\left( \bigcap_{d=0}^{n-1}  \bfW_n^d({\omega}_{d})^* 
  \right) &=& \sup_{\bP \in \mathcal{F} } \overline{\delta}(\bP).
\end{eqnarray*}
Here, $\underline{\delta}(\bP)$ and $\overline{\delta}(\bP)$ are  the lower and upper average contraction rates, defined in \eqref{ud} and \eqref{od}.\\
 
Hence, we get a good lower bound for dimensions if we find Roy-systems $\bP \in \mathcal{F}$ with a high lower or upper average contraction rates. The idea is that the contraction rate is maximal and equal to $n$ if and only if the first component is growing, in particular when all the components are equal and growing with slope $\frac{1}{n+1}$. \\
 
 For upper bounds, we consider potential functions. Following techniques of \cite{VarPrinc}.\\
 
 \subsection{Proof of Theorem \ref{Hwd}}

To get the  best lower bound for the Hausdorff dimension of the Diophantine set, the idea is to have patterns insuring the requested Diophantine properties in between long intervals where all the components are equal and the local contraction rate maximal.\\
Fix the dimension $n\ge 2$, choose $0\le d \le n-1$ and $\infty > \omega_{d}\ge \frac{d+1}{n-d}$. Fix $1>\eps>0$, and define the sequence $(Q_k)_{k\ge0}$ by $Q_0=1$ and the recurrence
\[ Q_{2k+1} = \frac{(n-d) \omega_{d}}{d+1}Q_{2k} \textrm{ and } Q_{2k+2} = \frac{1}{\eps}Q_{2k+1}.\]
We define a Roy-system $\bP_\eps$ on $[1, \infty)$. On the intervals of the form $[Q_{2k+1}, Q_{2k+2}]$, $\bP_\eps$ is defined by $P_1(q)= \cdots = P_{n+1}(q)=\frac{q}{n+1}$ and the local contraction rate is maximal : $\delta([Q_{2k+1}, Q_{2k+2}])=n$. Note that if $\omega_{d}= \frac{d+1}{n-d}$, then $Q_{2k+1} = Q_{2k}$ for all $k$ and $\bP_\eps$ is defined by $P_1(q)= \cdots = P_{n+1}(q)=\frac{q}{n+1}$ on $[1, \infty)$.\\
On the intervals of the form $[Q_{2k},Q_{2k+1}]$, $\bP_\eps$ is defined by the following pattern (see Figure \ref{figC}).

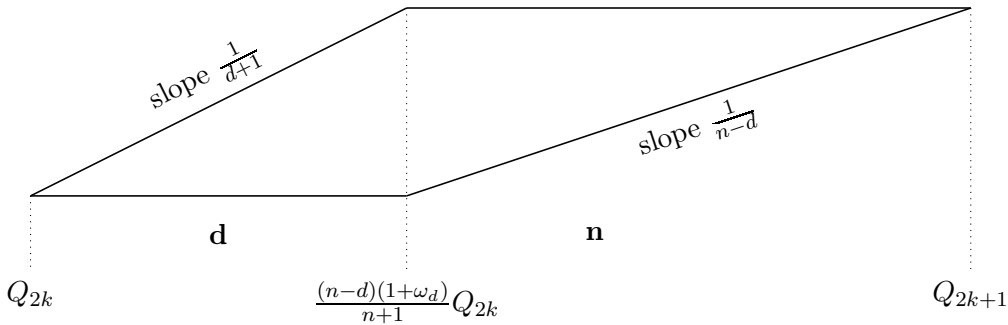
\begin{figure}[!h] 
 \begin{center}
 \begin{tikzpicture}[scale=0.5]
 
\draw[black, semithick] (10,5)--(0,0) node [midway, above, sloped] {slope $\frac{1}{d+1}$};
\draw[black, semithick] (10,0)--(0,0) ;
\draw[black, semithick] (10,5)--(25,5);
\draw[black, semithick] (10,0)--(25,5) node [midway, below, sloped] {slope $\frac{1}{n-d}$};

\draw[black, dotted] (10,5)--(10,-2)node [below,black] {$\frac{(n-d)(1+\omega_{d})}{n+1}Q_{2k}$};
\draw[black, dotted] (0,0)--(0,-2)node [below,black] {$Q_{2k}$};
\draw[black, dotted] (25,5)--(25,-2)node [below,black] {$Q_{2k+1}$};

\draw (5,-1) node {\textbf{d}};
\draw (15,-1) node {\textbf{n}};

 \end{tikzpicture}
 \end{center}
 \caption{Roy-system $\bP_\eps$ on the interval $[Q_{2k}, Q_{2k+1}]$ for the Hausdorff dimension of the Diophantine sets. Local contraction rates are in bold.  }\label{figC}
 \end{figure}

At $Q_{2k}$, we have $P_1(Q_{2k})= \cdots = P_{n+1}(Q_{2k})=\frac{Q_{2k}}{n+1}$.

Then, the $d+1$ components $P_{n-d+1}, \ldots, P_{n+1}$ grow together with slope $\frac{1}{d+1}$ until reaching the value $\frac{(n-d) \omega_d}{(d+1)(n+1)}$ at $q=\frac{(n-d)(1+\omega_{d})}{n+1}Q_{2k}$. On this interval of length $|I|=\frac{(n-d)(\omega_d+1)-(n-1)}{n+1}$, the contraction rate is $\delta(I)=d$.

Then, the $n-d$ components $P_1, \ldots , P_{n-d}$ grow together with slope $\frac{1}{n-d}$ until reaching the value $P_{d+1}= \cdots = P_{n+1}= \frac{(n-d) \omega_d}{(d+1)(n+1)}$ at $q=Q_{2k}\frac{(n-d) \omega_{d}}{d+1} = Q_{2k+1}$. On this interval the contraction rate is $\delta(I)=n$.\\

Since $P_1+ \cdots +P_{n-d}$ has slope $0$ on the intervals $[Q_{2k},\frac{(n-d)(1+\omega_{d})}{n+1}Q_{2k}]$ and $1$ otherwise, we have
\begin{eqnarray*}
\liminf_{q\to \infty} \frac{P_1(q) + \cdots + P_{n-d}(q)}{q} &=& \lim_{k\to \infty} \frac{P_1(\frac{(n-d)(1+\omega_{d})}{n+1}Q_{2k}) + \cdots + P_{n-d}(\frac{(n-d)(1+\omega_{d})}{n+1}Q_{2k})}{\frac{(n-d)(1+\omega_{d})}{n+1}Q_{2k}}\\
&=& \frac{1}{1+\omega_d}.
\end{eqnarray*}

Using Roy's theorem, $\bP_\eps$ provides points in $\bfW_n^d(\omega_d)$.\\

Since
\begin{eqnarray*}
\Delta(\bP_\eps,Q_{2k}) &=& \frac{n(Q_{2k}-Q_{2k-1})+(Q_{2k-1}-1)\Delta(\bP_\eps,Q_{2k-1})}{Q_{2k}-1}\\
 &=& \frac{n(1-\eps)Q_{2k}-(\eps Q_{2k}-1)\Delta(\bP_\eps,Q_{2k})}{Q_{2k}-1}\\
 &\sim_{k\to \infty}& n-\eps(n+ \Delta(\bP_\eps,Q_{2k-1})),
 \end{eqnarray*}
for every $1>\eps>0$, \[\overline{\delta}(\bP_\eps) \ge \limsup_{k\to \infty} \Delta(\bP_\eps, Q_{2k}) = n-\eps(n+\Delta(\bP_\eps,Q_{2k-1})) \to_{\eps\to 0} n.\] 
The local contraction rate is always bounded above by $n$, so we deduce the equality for the packing dimension.\\

The average contraction rate $\Delta(\bP_\eps,Q)$ reaches its local minima at $\frac{(n-d)(1+\omega_{d})}{n+1}Q_{2k}$. Indeed, the local contraction rate is $d$ on the intervals $[Q_{2k},\frac{(n-d)(1+\omega_{d})}{n+1}Q_{2k}]$ and $n$ otherwise. Hence,
\begin{eqnarray*}
 \liminf_{Q\to \infty} \Delta(\bP_\eps, Q) &=&  \liminf_{k\to \infty} \Delta(\bP_\eps, \frac{(n-d)(1+\omega_{d})}{n+1}Q_{2k})\\
 &=& \liminf_{k\to \infty}  \frac{(Q_{2k}-1)\Delta(\bP_\eps,Q_{2k}) + d(\frac{(n-d)(1+\omega_{d})}{n+1}Q_{2k} -Q_{2k})}{\frac{(n-d)(1+\omega_{d})}{n+1}Q_{2k}-1}\\
 &\sim_{\eps\to0}&  d +\frac{n+1}{1+\omega_d}.\\
\end{eqnarray*}
The lower bound for the Hausdorff dimension follows. If $\omega_d=\infty$, we replace the value of $\omega_d$ by $w_k$ in the construction of each interval $[Q_{2k},Q_{2k+1}]$, where $w_k\to_{k\to \infty}\infty$. All the computation holds, and we get the same lower bound for the Hausdorff dimension.\\

For the upper bound, we consider the potential function \[\Phi(q) =dq +(n+1)(P_1(q) + \cdots + P_{n-d}(q)).\]

\begin{lemma}\label{potential}
Let $I$ be an interval of linearity for $\bP$. Then for $q\in I$,
\[ \Phi'(q) \ge \delta(I).\]
\end{lemma}

\begin{proof}[Proof of Lemma \ref{potential}]
Fix an interval of linearity of $\bP$. Consider $k$ to be the minimal index such that $P_k'>0$ on $I$, and $j$ the maximal index such that $P_k= \cdots = P_{k+j}$. We recall that by definition $\delta(I)=n-k+1$. With these definitions, the derivatives on $I$ are $P'_k= \cdots = P'_{k+j} = \frac{1}{j+1}$ and $P'_i=0$ otherwise. Hence,
\begin{eqnarray*}
\Phi'(q) &=& d +(n+1)(P'_1(q) + \cdots + P'_{n-d}(q)) = d +(n+1)(P'_k(q) + \cdots + P'_{\min(n-d,k+j)}(q))\\
 &=&d +\frac{n+1}{j+1}(\min(n-d,k+j)-k+1)
\end{eqnarray*}
If $j\ge n-d-k$, then $\min(d,k+j)-k+1=n-d-k+1$ and as $j\le n$ we deduce
\[\Phi'(q) = d + \frac{n+1}{j+1}(n-d-k+1) \ge d+(n-d-k+1) \ge \delta(I).\]
If $j< n-d-k$, then $\min(d,k+j)-k+1=j+1$ and 
 \[\Phi'(q) = d + \frac{n+1}{j+1}(j+1) = d+n+1 \ge n \ge \delta(I).\]
\end{proof}

Integrating the inequality of Lemma \ref{potential} gives 
\[ dq +(n+1)(P_1(q) + \cdots + P_{n-d}(q)) \ge q \Delta(q).\]
Dividing by $q$ and then taking the $\liminf_{q\to\infty}$ gives
\[ \underline{\delta}(\bP)\le d + \frac{n+1}{1+\omega_{d}}.\]
This finish the proof of Theorem \ref{Hwd}.\qed

\subsection{Proofs of Theorems \ref{Hcap} and \ref{Pcap} }
 
 Fix $n\ge 2$, and choose $\omega_0, \ldots \omega_{n-1}$ satisfying the going up and going down inequalities \eqref{gugd}.\\
 
 Define $a_1=\frac{1}{1+\omega_{n-1}}, \quad a_d = \frac{1}{1+\omega_{n-d}} - \frac{1}{1+\omega_{n-d+1}}, \quad a_{n+1}=1 - \sum_{i=1}^{n}a_n = \frac{\omega_0}{1+\omega_0}$, where $1/\infty=0$. It follows from \eqref{gugd} that $a_d$ is increasing with $d$. See also \cite{Royspec}. Assume first that $\omega_{n-1}\ne \infty$ and thus $a_1>0$, we treat this pathological case later.\\
 
 For $1>\eps > 0$, we choose $Q_0=1$, and define the sequence $(Q_k)_{k\ge0}$ by $Q_0=1$ and the recurrence 
\[ Q_{2k+1} = \frac{a_{n+1}}{a_0}Q_{2k} \textrm{ and } Q_{2k+2} = \frac{1}{\eps}Q_{2k+1}.\]

Define $q_j$ for $0\le j \le 2n$ by
\begin{eqnarray*}
q_d &=& a_1 + \cdots + a_d + (n+1-d)a_{k+1} \quad \textrm{ for } 0\leq d \leq n,\\
q_{n+d} &=& a_{n+1} + \cdots + a_{d+2} + (1+d)a_{d+1}  \quad \textrm{ for } 0\leq d \leq n.
\end{eqnarray*}
Note that $q_n = a_1 + \cdots + a_{n+1}=1$.\\
To simplify formulae, we set $\alpha_i=\frac{a_i}{q_0}$ for $1\le i \le n+1$ and $\rho_i=\frac{q_i}{q_0}$ for $0\le i \le 2n$. Note that $\rho_0=1$.\\

We define a Roy-system $\bP_\eps$ on $[Q_0,\infty)$, by defining it on each subinterval $[Q_k,Q_{k+1}]$. On the intervals of the form $[Q_{2k+1},Q_{2k+2}]$, we set $\bP_\eps$ by $P_1(q) = \cdots = P_{n+1}(q) = \frac{q}{n+1}$, where the local contraction rate is maximal : $\delta([Q_{2k+1},Q_{2k+2}])=n$. Note that if $\omega_{n-1}=n$, then $a_1 = \cdots = a_{n+1}= \frac{1}{n+1}$, and for all $k$ we have $Q_{2k}=Q_{2k+1}$, meaning that $\bP_\eps$ is defined by $P_1(q) = \cdots = P_{n+1}(q) = \frac{q}{n+1}$ on $[Q_0,\infty)$. \\

On the intervals of the form $[Q_{2k},Q_{2k+1}]$, we construct a pattern illustrated by Figure \ref{figD}.\\ 

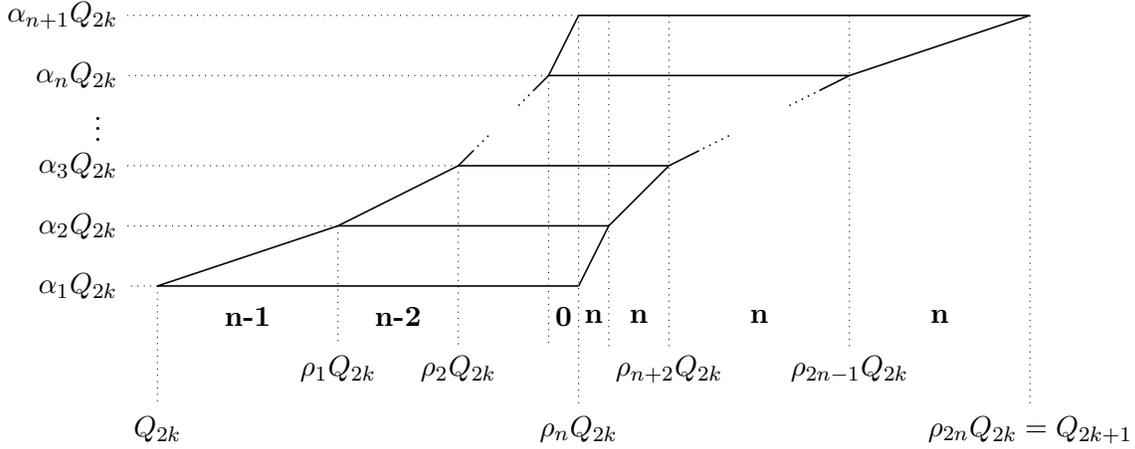
\begin{figure}[!h] 
 \begin{center}
 \begin{tikzpicture}[scale=0.4]
 
\draw[black, semithick] (6,2)--(0,0) node [midway, above, sloped] {};
\draw[black, semithick] (6,2)--(10,4);
\draw[black, semithick] (6,2)--(15,2);
\draw[black, semithick] (14,0)--(0,0);
\draw[black, semithick] (13,7)--(14,9);
\draw[black, semithick] (14,0)--(15,2);
\draw[black, semithick] (15,2)--(17,4);
\draw[black, semithick] (10,4)--(17,4);
\draw[black, semithick] (13,7)--(23,7);
\draw[black, semithick] (23,7)--(29,9);
\draw[black, semithick] (14,9)--(29,9);

\draw[black, semithick] (10,4)--(10.5,4.5);
\draw[black, dotted, semithick] (11,5)--(10.5,4.5);
\draw[black, dotted, semithick] (12,6)--(12.5,6.5);
\draw[black, semithick] (13,7)--(12.5,6.5);

\draw[black, semithick] (17,4)--(18,4.5);
\draw[black, dotted, semithick] (19,5)--(18,4.5);
\draw[black, dotted, semithick] (21,6)--(22,6.5);
\draw[black, semithick] (23,7)--(22,6.5);

\draw[black, dotted] (0,0)--(0,-4)node [below,black] {$Q_{2k}$};
\draw[black, dotted] (6,2)--(6,-2)node [below,black] {$\rho_1Q_{2k}$};
\draw[black, dotted] (10,4)--(10,-2)node [below,black] {$\rho_2Q_{2k}$};
\draw[black, dotted] (13,7)--(13,-2); 
\draw[black, dotted] (14,9)--(14,-4)node [below,black] {$\rho_nQ_{2k}$};
\draw[black, dotted] (15,9)--(15,-2);
\draw[black, dotted] (17,9)--(17,-2)node [below,black] {$\rho_{n+2}Q_{2k}$};
\draw[black, dotted] (23,9)--(23,-2)node [below,black] {$\rho_{2n-1}Q_{2k}$};
\draw[black, dotted] (29,9)--(29,-4)node [below,black] {$\rho_{2n}Q_{2k}=Q_{2k+1}$};

\draw[black, dotted] (14,9)--(-1,9)node [left,black] {$\alpha_{n+1}Q_{2k}$};
\draw[black, dotted] (13,7)--(-1,7)node [left,black] {$\alpha_{n}Q_{2k}$};
\draw[black, dotted] (10,4)--(-1,4)node [left,black] {$\alpha_{3}Q_{2k}$};
\draw[black, dotted] (6,2)--(-1,2)node [left,black] {$\alpha_{2}Q_{2k}$};
\draw[black, dotted] (0,0)--(-1,0)node [left,black] {$\alpha_{1}Q_{2k}$};
\draw (-1.5,5.5) node  [left] {$\vdots$};

\draw (3,-1) node {\textbf{n-1}};
\draw (8,-1) node {\textbf{n-2}};
\draw (13.5,-1) node {\textbf{0}};
\draw (14.5,-1) node {\textbf{n}};
\draw (16,-1) node {\textbf{n}};
\draw (20,-1) node {\textbf{n}};
\draw (26,-1) node {\textbf{n}};

 \end{tikzpicture}
 \end{center}
 \caption{Roy-system $\bP_\eps$ on the interval $[Q_{2k}, Q_{2k+1}]$ for the Hausdorff dimension of intersections of Diophantine sets. Local contraction rates are in bold. }\label{figD}
 \end{figure}

We start with $P_1(Q_{2k}) = \cdots = P_{n+1}(Q_{2k}) = \frac{Q_{2k}}{n+1}=\alpha_1Q_{2k}$.\\
Then, for $0\le d \le n$, on the interval $[\rho_dQ_{2k}, \rho_{d+1}Q_{2k} ]$ the $(n-d)$ components $P_{d+2}, \ldots , P_{n+1}$ grow together with slope $\frac{1}{n-d}$. On this interval of derivability of length $|I_d| = (\rho_{d+1}-\rho_{d})Q_{2k}$, the contraction rate is $\delta(I_d)=n-1-d$. At $q=\rho_{d+1}Q_{2k}$, we have
\begin{eqnarray*}
P_i(\rho_{d+1}Q_{2k}) &=& \alpha_i Q_{2k}, \textrm { for } 1\le i \le d ,\\
P_{d+1}(\rho_{d+1}Q_{2k}) &=& \cdots \; = P_{n+1}(\rho_{d+1}Q_{2k}) = \alpha_dQ_{2k}.
\end{eqnarray*}
Then, for $0\le d \le n$, on the interval $[\rho_{n+d}Q_{2k}, \rho_{n+d+1}Q_{2k} ]$, the $d+1$ components $P_1, \ldots , P_{d+1}$ grow together with slope $\frac{1}{d+1}$. On this interval of derivability of length $|I'_d| = (\rho_{n+d+1}-\rho_{n+d})Q_{2k}$, the contraction rate is $\delta(I'_d)=n$. At $q=\rho_{n+d+1}Q_{2k}$, we have
\begin{eqnarray*}
P_i(\rho_{n+d+1}Q_{2k}) &=& \alpha_i Q_{2k}, \textrm { for } d\le i \le n+1, \\
P_{1}(\rho_{n+d+1}Q_{2k}) &=& \cdots \; = P_{d+1}(\rho_{n+d+1}Q_{2k}) = \alpha_{d+1}Q_{2k}.
\end{eqnarray*}

This satisfies the conditions of a Roy-system.\\

Here, the main pattern is appearing with low frequency on the intervals $[Q_{2k},Q_{2k+1}]$ and is essentially the one constructed by Roy in \cite{Royspec}. He showed in detail that
\begin{eqnarray*}
\liminf_{q\to \infty} \frac{P_1(q) + \cdots + P_{d}(q)}{q}&=& \lim_{k\to \infty}\frac{P_1(\rho_nQ_{2k}) + \cdots + P_{d}(\rho_nQ_{2k})}{\rho_nQ_{2k}} = \frac{a_1 + \cdots + a_d}{q_n} = \frac{1}{1+\omega_{n-d}},\\
\limsup_{q\to \infty} \frac{P_1(q) + \cdots + P_{d}(q)}{q} &=& \frac{d}{n+1}.
\end{eqnarray*}
Hence, the points $\btheta$ associated to $\bP$ by Roy's theorem are in the set $\bigcap_{d=0}^{n-1} \bfW_n^d(\omega_{d})^*
$, and the upper (resp. lower) average contraction rate of $\bP_\eps$ is a lower bound of the Hausdorff (resp. packing) dimension of this set.\\


First, we study the upper average contraction rate, and thus the packing dimension.

\begin{eqnarray*}
\Delta(\bP_\eps,Q_{2k}) &=& \frac{n(Q_{2k}-Q_{2k-1})+(Q_{2k-1}-1)\Delta(\bP_\eps,Q_{2k-1})}{Q_{2k}-1}\\
 &=& \frac{n(1-\eps)Q_{2k}-(\eps Q_{2k}-1)\Delta(\bP_\eps,Q_{2k})}{Q_{2k}-1}\\
 &\sim_{k\to \infty}& n -\eps(n+\Delta(\bP_\eps,Q_{2k-1})).
 \end{eqnarray*}
 Hence for every $1>\eps>0$, \[\overline{\delta}(\bP_\eps) \ge \limsup_{k\to \infty} \Delta(\bP_\eps, Q_{2k}) = n -\eps(n+\Delta(\bP_\eps,Q_{2k-1})) \to_{\eps\to 0} n.\] Since the local contraction rate is always bounded above by $n$, we deduce that the packing dimension is full. This proves Theorem \ref{Pcap}.\qed\\
 
 We now study the lower average contraction rate, and thus the Hausdorff dimension.\\
 
 The average contraction rate $\Delta(\bP_\eps,t)$ reaches its local minima at $\rho_nQ_{2k}$, provided $\eps \le \frac{1}{2n}$. Indeed, starting with the estimate of $\Delta(\bP_\eps,Q_{2k})$, one can prove that $\Delta(\bP_\eps,q)$ is decreasing on $[Q_{2k},\rho_nQ_{2k}]$, as the local contraction rate stays lower than its value. Thus,
 \begin{eqnarray*}
\underline{\delta}(\bP_\eps)&=& \liminf_{t\to \infty} \Delta(\bP_\eps,t) = \liminf_{k\to\infty} \Delta(\bP_\eps, \rho_nQ_{2k})\\
&=& \frac{\rho_0Q_{2k} \Delta(\bP_\eps,Q_{2k}) + \sum_{d=0}^{n}(n-d-1)(\rho_{d+1}-\rho_d)Q_{2k} }{\rho_n Q_{2k}}\\
&\sim_{k\to \infty, \eps\to 0}& \frac{q_0+ \cdots + q_{n-1}}{q_n} = 2\sum_{d=1}^{n}(n+1-i)a_i = 2 \sum_{d=1}^{n} \sum_{j=1}^{d}a_j = 2\sum_{d=0}^{n-1} \frac{1}{1+\omega_d}.
 \end{eqnarray*}
 Hence the expected lower bound for the Hausdorff dimension.\\
 
 We now deal with the pathological case, based on a light adjustment of the previous construction. The main idea is that the parameters $(a_d)_{d=1}^{n+1}$ do not have to be constant, and can vary with $k$. The pathology is when $\omega_{n-1}=\infty$, then we would have $a_1=0$ and our construction collapse. In this case, we replace $a_1$ by a sequence $(a_{1,k})_{k\ge0}$ that strictly decreases to $a_{1,k}\to0$. Then for $2\le d \le n+1$ we set $a_{d,k}=\max(a_d, a_{d-1,k})$. Unfortunately, they do not sum to $1$, so we need to renormalize. We define
\[\tilde{a}_{d,k} = \frac{a_{d,k}}{a_{1,k} + \cdots + a_{n+1,k}} \to_{k\to \infty} a_d.\]
We now consider the sequence $(Q_{k})_{k\ge 0}$ defined by $Q_0 =1$ and 
\[ Q_{2k+1} = \frac{\tilde{a}_{n+1,k}}{\tilde{a}_{0,k}}Q_{2k} \textrm{ and } Q_{2k+2} = \frac{1}{\eps}Q_{2k+1}.\]
Mutatis mutandis $a_d$ by $\tilde{a}_{d,k}$, we follow the previous construction on the intervals $[Q_k,Q_{k+1}]$. Previous computations holds, and the same conclusion follows.\\

 We now need an upper bound for the Hausdorff dimension. We consider the potential function
 \[ \Phi(q) = 2\sum_{d=1}^{n} \sum_{j=1}^{d}P_j(q)=2\sum_{d=1}^{n} (n+1-j) P_j(q).\]
 
 \begin{lemma}\label{potential3}
Let $I$ be an interval of linearity for $\bP$. Then for $q\in I$, 
\[ \Phi'(q) \ge \delta(I).\]
\end{lemma}

\begin{proof}[Proof of Lemma \ref{potential3}]
Fix an interval of linearity of $\bP$. Consider $k$ to be the minimal index such that $P_k'>0$ on $I$, and $j$ the maximal index such that $P_k= \cdots = P_{k+j}$. We recall that by definition $\delta(I)=n-k+1$.\\
With these definitions, the derivatives on $I$ are  $P_k'= \cdots = P'_{k+j}= \frac{1}{j+1}$ and $P'_i=0$ otherwise. Hence,
\begin{eqnarray*}
\Phi'(q) &=& 2\sum_{d=1}^{n} (n+1-j) P'_j(q) \\
  &=& \frac{2}{j+1}\sum_{i=1}^{j} (n +1 -k -j+i)\\
  &=&\frac{2}{j+1}\left( (j+1)(n+1-k-j) + \sum_{i=1}^{j}i\right)\\
  &=& 2\left((n+1-k-j) +\frac{j}{2} \right) \ge \delta(I),
\end{eqnarray*}
since $j \le n+1-k =\delta(I)$.
\end{proof}

Integrating the inequality of Lemma \ref{potential3} gives 
\[ 2\sum_{d=1}^{n} \sum_{j=1}^{d}P_j(q) \ge q \Delta(q).\]
Dividing by $q$ and then taking the $\liminf_{q\to\infty}$ gives
\[ \underline{\delta}(\bP)\le 2\sum_{d=0}^{n-1} \frac{1}{1+\omega_d}.\]
This provides the upper bound for the Hausdorff dimension. Since the bound is a decreasing function in term of $\omega_d$, Theorem \ref{Hcap} follows.\qed

\begin{remark}If we look carefully at the constructions providing lower bounds, we have the extra condition that $\hat{\omega}_{n-1}=n$. Hence, we proved the slightly stronger lower bounds 
\begin{eqnarray*}
\dim_P\left( \bigcap_{d=0}^{n-1} \bfW_n^d({\omega}_{d})^* \cap \hat{\bfW}_n^{n-1}(n)^*  \right) & \ge& n,\\
\dim_H\left( \bigcap_{d=0}^{n-1} \bfW_n^d({\omega}_{d})^* \cap \hat{\bfW}_n^{n-1}(n)^*  \right)  &\ge& 2\sum_{d=0}^{n-1} \frac{1}{1+\omega_d}.
\end{eqnarray*}

Here, $\hat{\bfW}_n^d(\hat{\omega}_{d})^*=\{\btheta\in \RR^n | \hat{\omega}_{d}(\btheta)=\hat{\omega}_{d}\}$ is a uniform Diophantine set.\\

Since the bound is a decreasing function in term of $\omega_d$, we have the slightly stronger 

\begin{theorem}
Fix the dimension $n\ge 2$, choose a $n$-tuple of real numbers $\omega_{0}, \ldots , \omega_{n-1}$ satisfying the \emph{going-up} and \emph{going down} relations \eqref{gugd}. Then
\begin{eqnarray*}
\dim_P\left( \bigcap_{d=0}^{n-1} \bfW_n^d({\omega}_{d})^* \cap \hat{\bfW}_n^{n-1}(n)^*  \right) &=& \dim_P\left( \bigcap_{d=0}^{n-1} \bfW_n^d({\omega}_{d}) \cap \hat{\bfW}_n^{n-1}(n)^* \right) \\
= \dim_P\left( \bigcap_{d=0}^{n-1} \bfW_n^d({\omega}_{d})^*\right) &=& \dim_P\left( \bigcap_{d=0}^{n-1} \bfW_n^d({\omega}_{d}) \right) =n
\end{eqnarray*}
and
\begin{eqnarray*}
\dim_H\left( \bigcap_{d=0}^{n-1} \bfW_n^d({\omega}_{d})^* \cap \hat{\bfW}_n^{n-1}(n)^*  \right) &=& \dim_H\left( \bigcap_{d=0}^{n-1} \bfW_n^d({\omega}_{d}) \cap \hat{\bfW}_n^{n-1}(n)^* \right) \\ \notag
= \dim_H\left( \bigcap_{d=0}^{n-1} \bfW_n^d({\omega}_{d})^*\right) &=& \dim_H\left( \bigcap_{d=0}^{n-1} \bfW_n^d({\omega}_{d}) \right)\\
&=& 2\left( \sum_{d=0}^{n-1 }\frac{1}{1+\omega_d}\right).
\end{eqnarray*}

\end{theorem}
\end{remark}

\end{document}